\documentclass{birkjour}
 \newtheorem{thm}{Theorem}[section]
 
 \newtheorem{lem}[thm]{Lemma}
 \newtheorem{prop}[thm]{Proposition}
 \theoremstyle{definition}
 \newtheorem{defn}[thm]{Definition}
 \theoremstyle{remark}
 \newtheorem{rem}[thm]{Remark}

\usepackage{color}
 \numberwithin{equation}{section}

\def\r{\mathbb R}

\def\h{\mathbb H}

\begin{document}

%
%
%
%
%
%
%
%
%

\title{A new family of translating solitons in hyperbolic space}

\author{Antonio Bueno}

\address{Departamento de Ciencias\\  Centro Universitario de la Defensa de San Javier. 30729 Santiago de la Ribera, Spain}

\email{antonio.bueno@cud.upct.es}

\author{Rafael L\'opez}
\address{Departamento de Geometr\'{\i}a y Topolog\'{\i}a\\  Universidad de Granada. 18071 Granada, Spain}
\email{rcamino@ugr.es}


\date{}
 ------------------------------------------------------------------

\begin{abstract}
If   $\xi$ is a Killing vector field of the hyperbolic space $\h^3$ whose flow are parabolic isometries, a surface $\Sigma\subset\h^3$ is a $\xi$-translator  if its mean curvature $H$ satisfies $H=\langle N,\xi\rangle$, where $N$ is the unit normal of $\Sigma$. We classify all $\xi$-translators invariant by a one-parameter group of rotations of $\h^3$,  exhibiting the existence of a new family of grim reapers. We use these grim reapers   to prove the non-existence of closed $\xi$-translators.
\end{abstract}

\maketitle
\section{Introduction}
 
Let  $\Sigma$ be an orientable smooth surface and $\Psi:\Sigma\to\h^3$ an isometric immersion   in hyperbolic space  $\h^3$. The mean curvature flow (MCF in short)  is a differentiable map ${\bf \Psi}:\Sigma\times [0,T)\to\h^3$ such that if ${\bf \Psi}_t={\bf \Psi}(-,t)$, then ${\bf \Psi}_0=\Psi$  and $\frac{\partial{\bf \Psi}_t}{\partial t}  =H({\bf \Psi}_t)N({\bf \Psi}_t)$,  where $H({\bf \Psi}_t)$ and $N({\bf \Psi}_t)$ are the mean curvature and the unit normal of ${\bf \Psi}_t$ respectively. See \cite{an,cm}. Our interest are  those surfaces whose shape evolves along the MCF by translations along a direction of $\h^3$. These surfaces are called translating solitons of the MCF, or translators for short. In Euclidean space $\r^3$, translators  $\Sigma$ along a direction $\textbf{v}\in\r^3$ are characterized by  $H=\langle N,\textbf{v}\rangle$, where $H$ and $N$ are the mean curvature and the unit normal of $\Sigma$, respectively. Translators also  appear  in the singularity theory of the MCF after a blow-up near type II singularities, according to Huisken and Sinestrari \cite{hs}.

In the hyperbolic space, the same notion of translator can be stated by taking the translations of $\h^3$ determined from a Killing vector field $X\in\mathcal{X}(\h^3)$.   Geometrically, the shape of any such translator does not change during the evolution by motion along the one-parameter group of translations generated by $X$.

\begin{defn} Let $X\in\mathcal{X}(\h^3)$ be a Killing vector field whose flow are translations of $\h^3$. An immersed surface $\Sigma$ in $\h^3$ is an $X$-translator if its mean curvature $H$ satisfies 
\begin{equation}\label{eq1}
H=\langle N,X\rangle.
\end{equation}
\end{defn} 

In hyperbolic space there are parabolic translations and hyperbolic translations. Translators of the MCF whose shape evolves by hyperbolic translations have been recently studied in \cite{li}. In this paper, we investigate $\xi$-translators, where $\xi$ is a Killing vector field that generates parabolic translations. We point out that these translators in the MCF theory of $\h^3$ have not been previously studied in the literature. 

Consider the upper half-space model of $\h^3$, that is $(\r^3_{+},\langle,\rangle)$, where $\r_+^3=\{(x,y,z)\in\r^3:z>0\}$ and $\langle,\rangle$ is  the hyperbolic metric 
$$\langle,\rangle =\frac{1}{z^2}\langle,\rangle_e,$$ 
where $\langle,\rangle_e=dx^2+dy^2+dz^2$ is the Euclidean metric of $\r^3$. The ideal boundary $\partial_\infty\h^3$ is the one-compactification of the plane of equation $z=0$. Define the Killing vector field
\begin{equation}\label{x2}
\xi=a\partial_x+b\partial_y, \qquad a,b\in\r.
\end{equation}
The flow of isometries generated by $\xi$ are parabolic translations of $\h^3$, which in the upper half-space model correspond to horizontal Euclidean translations in the direction of  $(a,b,0)$. 

In this paper we classify all $\xi$-translators invariant by a one-parameter group of rotations. In hyperbolic space $\h^3$ there are three types of group of rotations. Following  Do Carmo and   Dajczer  \cite{dd}, there are parabolic rotations (a double point of $\h^3_\infty$ is fixed), hyperbolic rotations (two points of $\partial_\infty\h^3$ are fixed) and spherical rotations (a geodesic is pointwise fixed).  Consequently, there are three types of rotational surfaces called, respectively, parabolic, hyperbolic  and spherical rotational surfaces. In the literature,   parabolic and hyperbolic rotations are also called   parabolic and hyperbolic translations.

The first result gives a classification for hyperbolic and spherical rotational $\xi$-translators.

\begin{thm}\label{t1} 
\begin{enumerate}
\item Let $\Sigma$ be a hyperbolic rotational $\xi$-translator. Then, $\Sigma$ is a totally geodesic plane containing the two points at $\partial_\infty\h^3$ fixed by the hyperbolic rotations.
\item There are no spherical rotational $\xi$-translators.
\end{enumerate}
\end{thm}

More interesting is the family of parabolic $\xi$-translators.  After an isometry of $\h^3$, and from the Euclidean viewpoint, a parabolic rotational surface is a cylindrical surface of $\r^3_{+}$ whose rulings are horizontal lines. Such a surface can be parametrized as $\Psi(s,t)=(x(s),t,z(s))$, where $(x(s),0,z(s))$ is the generating curve. By analogy with the Euclidean case, we give the following definition.

\begin{defn}
A $\xi$-grim reaper is a parabolic $\xi$-translator.
\end{defn}

The main result in this paper is the following classification of the $\xi$-grim reapers.

\begin{thm}\label{t2}
Let be $\xi=a\partial_x+b\partial_y,\ a,b\in\r$, and $\Sigma$ a $\xi$-grim reaper.
\begin{itemize}
\item If $a=0$, then $H=0$ and $\Sigma$ is either a vertical plane parallel to $\xi$, or belongs to a one-parameter family $\mathcal{G}^0(z_0)$. The generating curve of each $\mathcal{G}^0(z_0)$ is a strictly concave graph on $x$-axis, intersecting orthogonally the $x$-axis at two points, and whose maximum height to the $x$-axis is $z_0$.
\item If $a\neq0$, then $\Sigma$ belongs to a one-parameter family $\mathcal{G}(z_0)$. The generating curve of each $\mathcal{G}(z_0)$ is a bi-graph on the $x$-axis, its maximum height to the $x$-axis is $z_0$, and both graphical components converge to the $x$-axis as $x\rightarrow\infty$.
\end{itemize}
\end{thm}

This paper is organized as follows. In Sec. \ref{sec2} we prove Th. \ref{t1}, while in Sec. \ref{sec3} we prove Th. \ref{t2}. It is of special relevance the proof of the case $a\neq0$, for which we study the qualitative properties of the generating curve of a $\xi$-grim reaper by means of a phase plane analysis. In Sec. \ref{sec4}, and under mild hypothesis, we prove the non-existence of  $\xi$-translators contained in the convex side of Killing cylinders and consequently, there are no closed $\xi$-translators.  Due to the variety of translators and grim reapers that can be defined in $\h^3$, an Appendix has been added to summarize the literature concerning it.

\section{Proof of Theorem \ref{t1}}\label{sec2}

In the upper half-space model of $\h^3$ there is a relation between the mean curvature $H$ of a surface $\Sigma$ and its Euclidean mean curvature $H_e$, when $\Sigma$ is regarded as surface isometrically immersed in $\r^3_{+}$. This relation is
\begin{equation}\label{mean}
H(x,y,z)=zH_e(x,y,z)+(N^e)_3(x,y,z),\quad (x,y,z)\in\Sigma,
\end{equation}
where $N^e$ is the   Euclidean unit normal of $\Sigma$ and the subindex $(\cdot)_3$ denotes the third coordinate of the vector. To prove Th. \ref{t1} we distinguish between hyperbolic and spherical rotational surfaces.

First, assume that $\Sigma$ is a hyperbolic rotational surface. From the Euclidean viewpoint, $\Sigma$ is invariant under the homotheties from a point of the plane $z=0$. After an isometry of $\h^3$, this point can be fixed to be the origin $O$ of $\r^3_{+}$. Then,   $\Sigma$ can be parametrized as the radial graph of a curve $\alpha(s)=(x(s),y(s),1)$, that is, 
$$\Psi(s,t)=t\alpha(s)=t(x(s),y(s),1),\quad s\in I\subset\r,\ t\in\r.$$
Suppose that $s$  is the Euclidean arc-length parameter.   The Euclidean mean curvature $H_e$ and unit normal $N^e$ are
$$H_e=\frac{| \alpha|_e^2\langle \alpha'\times\alpha,\alpha''\rangle_e}{2t(| \alpha|_e^2-\langle\alpha',\alpha\rangle_e^2)^{3/2}}, \quad N^e=\frac{\alpha'\times\alpha}{(| \alpha|_e^2-\langle\alpha',\alpha\rangle_e^2)^{1/2}}.$$
Since $N=zN^e$, then $\langle N,\xi\rangle=\frac{1}{t}\langle N^e,\xi\rangle_e$. Thus \eqref{eq1} is 
\begin{equation}\label{aa}
\frac{| \alpha|_e^2\langle \alpha'\times\alpha,\alpha''\rangle_e}{2(| \alpha|_e^2-\langle\alpha',\alpha\rangle_e^2)}+ (\alpha'\times\alpha)_3=\frac{\langle\alpha'\times \alpha, \xi\rangle_e}{t}.
\end{equation}
As this equation holds for any $t\in\r$, we deduce that the left hand-side of \eqref{aa} vanishes identically, that is, $H=0$. The same occurs for the right hand-side of \eqref{aa}, so  $ \langle\alpha'\times \alpha,\xi\rangle_e=0$, which implies that $\alpha$  is a horizontal line parallel to the vector $(a,b,0)$.  This proves that $\Sigma$ is a plane of $\r^3$ through $O$. In hyperbolic geometry, $\Sigma$ is a totally geodesic vertical plane if the curve $s\mapsto (x(s),y(s),0)$ passes through $O$,  or $\Sigma$ is an equidistant plane in other case. Because the mean curvature of an equidistant plane is not zero, the result holds. 

Now, we assume that $\Sigma$ is a spherical rotational surface, which up to an isometry can be then parametrized by 
 $$\Psi(s,t)=(x(s)\cos t,x(s)\sin t,z(s)),\quad s\in I\subset\r,\ t\in\r,$$
 where $\alpha(s)= (x(s),0,z(s))$ is the generating curve of $\Sigma$. We assume that $s$ is  the Euclidean arc-length parameter, hence $\alpha'(s)=( \cos\theta(s),0,\sin\theta(s))$ for some smooth function $\theta=\theta(s)$. A straightforward computation shows that Eq. \eqref{eq1} writes as
 $$\frac{z}{2}\left(\theta'+\frac{z'}{x}\right)+x'=-\frac{z'}{z}(a\cos t+b\sin t).$$
Since this equation holds for any $t$,   the functions $\{1,\cos t,\sin t\}$ are linearly independent and $a,b$ are not both identically zero, we deduce that the left hand-side is zero,  that is, $H=0$, and from the right hand-side, we have $z'(s)=0$ for all $s\in I$. This means that $\alpha$ is a horizontal Euclidean straight line, which yields that $\Sigma$ is a horosphere.  However, the mean curvature of a horosphere is $H=1$, obtaining a contradiction.

\section{Classification of the $\xi$-grim reapers}\label{sec3}
 
In the upper half-space model of $\h^3$, a parabolic rotational surface $\Sigma$ can be parametrized as a cylindrical surface of $\r^3_+$ whose rulings are horizontal lines. After a rotation about the $z$-axis, which only changes the constants $a$ and $b$ in \eqref{x2},   we can suppose that the direction of the rulings is $(0,1,0)$. Then, a parametrization for $\Sigma$ is
\begin{equation}\label{para}
\Psi(s,t)=(x(s),t,z(s)),\quad s\in I\subset\r, t\in\r,
\end{equation}
where  $\alpha(s)=(x(s),0,z(s))$ is   parametrized by the Euclidean arc-length. Then  $x'=\cos\theta$, $z'=\sin\theta$ for some smooth function $\theta=\theta(s)$. The Euclidean mean curvature and unit normal of $\Sigma$ are $H_e=\theta/2$ and $N^e=(-z',0,x')$. If $\Sigma$ is a $\xi$-grim reaper, according to \eqref{mean}, Eq.  \eqref{eq1} is 
\begin{equation}\label{eq2}
z\frac{\theta'}{2}+x'=-a\frac{z'}{z}.
\end{equation}
At this point, we distinguish if $a$ vanishes or not.

\subsection{Proof of Theorem \ref{t2}: case $a=0$}
In Th. \ref{t2} we assume  $a=0$, hence Eq. \eqref{eq2} implies  $H=0$.  Parabolic rotational surfaces of $\h^3$ with $H=0$   were classified in \cite{dd,go}. For completeness, we describe these surfaces. If $x'=0$ at some point, then the solution of  \eqref{eq2} is $\theta(s)=0$, $\alpha$ is a vertical line and $\Sigma$ is a totally geodesic plane parallel to $(0,1,0)$. Suppose now $x'(s)\not=0$ for all $s$. Then   $\alpha$ is a graph $z=z(x)$ on the $x$-axis,   and \eqref{eq2} is 
 \begin{equation}\label{eq-p1}
 \frac{z''}{1+z'^2}=-\frac{2}{z}.
 \end{equation}
 In particular, the curve $\alpha$ is a strictly concave graph. Multiplying \eqref{eq-p1}  by $z'$ and integrating, there is a positive constant $c>0$ such that $1+z'^2=cz^{-4}$.  Hence, it can be deduced that   $\alpha$   is  symmetric about a vertical line and that   $\alpha$   intersects orthogonally  the $x$-axis at two points. See Fig. \ref{fig1}. We denote this $\xi$-grim reaper by $\mathcal{G}^0(z_0)$ indicating the maximum height $z_0$ of $\alpha$ to the $x$-axis.

\begin{figure}[h]
\begin{center}
\includegraphics[width=.45\textwidth]{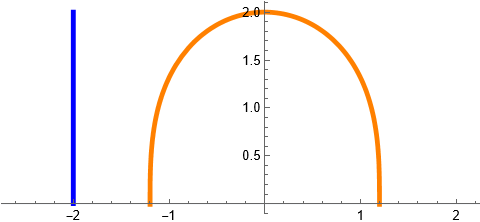}
\end{center}
\caption{The two types of generating curves of grim reapers when $\xi=\partial_y$.}\label{fig1}
\end{figure}

\begin{rem}\label{remark2} Eq. \eqref{eq-p1} appears in the context of singular minimal surfaces. More exactly, solutions of \eqref{eq-p1} are $-2$-catenaries following the terminology of \cite[Prop. 1]{lo} and their shapes are well known. See also \cite{dl}. 
\end{rem}

  \subsection{Proof of Theorem \ref{t2}: case  $a\neq0$}
  
Now, we suppose $a\not=0$ in \ref{t2}. After a reflection about a vertical plane and a hyperbolic translation we assume $a=1$, hence $\xi=\partial_x+b\partial_y$. From \eqref{eq2} the following system is fulfilled
\begin{equation}\label{eqs}
\left\{
\begin{split} 
x'&=\cos\theta  \\
z'&=\sin\theta \\
\theta'&=\displaystyle{-\frac{2}{z^2}\left(\sin\theta+z\cos\theta\right)}. 
\end{split}
\right.\end{equation}
Note that  the first equation of  \eqref{eqs} can be obtained from the second and third ones, and that the parameter $b$ does not appear. The phase plane of \eqref{eqs} is defined by $\Theta=\{(z,\theta)\colon z>0,\theta\in(-\pi,\pi)\}$, with coordinates denoting the height $z$ and the angle $\theta$. The  orbits $\gamma=(z,\theta)$ are the solutions of \eqref{eqs} when regarded in $\Theta$. By uniqueness, two distinct orbits cannot intersect, and the the Cauchy problem of \eqref{eqs} ensures the existence of an orbit passing through each point in $\Theta$. We point out that every orbit defines a generating curve $\alpha$ of a $\xi$-grim reaper and backwards. The following  result about  the phase plane  is straightforward, hence its proof is omitted.

\begin{prop}
The following properties of the phase plane hold.
\begin{enumerate}
\item If $\gamma(s)=(z(s),\theta(s))$ is an orbit, then $\overline{\gamma}(s)=(z(-s),\theta(-s)-\pi)$
is also an orbit. Consequently, we can restrict the coordinate $\theta$ to lie in $\theta\in(-\pi/2,\pi)$.
\item The curve $\Gamma:=\Theta\cap\{z=\Gamma(\theta)\}$, where $\Gamma(\theta)=-\tan\theta$, corresponds to points of $\alpha$ with vanishing Euclidean curvature. It only appears for $\theta\in(-\pi/2,0]\cup(\pi/2,\pi]$ and has the lines $\theta=0$ and $\theta=\pi/2$ as asymptotes. 
\item The curve $\Gamma$ and the line $\theta=0$ divide $\Theta$ into \emph{monotonicity regions} where each coordinate function of an orbit is strictly monotonous. See Fig. \ref{figfases}, left.
\end{enumerate}
\end{prop}

\begin{figure}[h]
\begin{center}
\includegraphics[width=.35\textwidth]{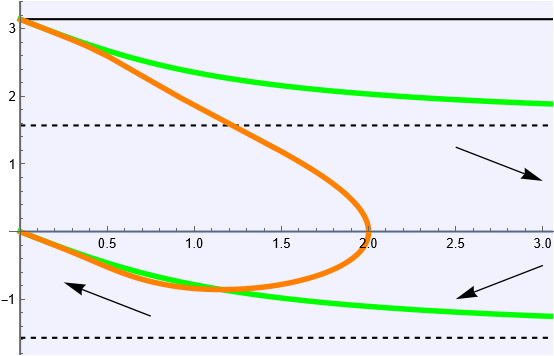}
\includegraphics[width=.45\textwidth]{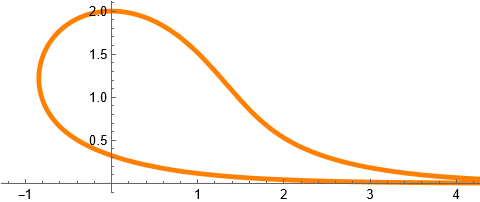}
\end{center}
\caption{Left: the phase plane and the orbit passing through $(2,0)$. Right: the corresponding   curve $\alpha$ of maximum height $2$.}
\label{figfases}
\end{figure}

The first result depicts the global behavior of the orbits and of the corresponding generating curves.

\begin{prop}\label{proporbitbehavior}
If $z_0>0$ there exists a unique orbit $\gamma_{z_0}$ passing through $(z_0,0)$, which converges to $z=0$ as $|s|$ increases. The corresponding generating curve $\alpha_{z_0}$ converges to $z=0$ and has Euclidean height $z_0$.
\end{prop}

\begin{proof}
Given $z_0>0$, consider the solution $\alpha_{z_0}$ of \eqref{eqs} with initial conditions $x(0)=0$, $z(0)=z_0$ and $\theta(0)=0$, and let $\gamma_{z_0}$ denote the corresponding orbit. The initial conditions yield that at $s=0$ the height function of $\alpha_{z_0}$ attains a local maximum. Indeed, $\theta'(0)=-2/z_0^2<0$, $z'(0)=0$ and for $s>0$ close enough to $s=0$ the orbit $\gamma_{z_0}$ lies in the region $\theta<0$ and $z>-\tan\theta$, hence $\theta'(s)<0$. By continuity, $\gamma_{z_0}$ ends up intersecting $\Gamma$ at some instant $s_0>0$ where $\theta'(s_0)=0$, hence $\alpha_{z_0}$ has vanishing curvature at $s=s_0$. Then, $\theta'(s)>0$ for $s>s_0$ since $\gamma_{z_0}$ lies in the region $z<-\tan\theta$. Consequently, $\gamma_{z_0}$ ends up converging to the boundary component $z=0$, and the motion of the orbits in the phase plane forbids $\gamma_{z_0}$ to converge to $(0,-\pi/2)$. On the other hand, for $s<0$ close enough to $s=0$, $\gamma_{z_0}$ lies in the region $\theta\in(0,\pi/2)$ and when $s<0$ further decreases it can behave in two ways: either $\gamma_{z_0}(s)\rightarrow (0,\theta_*)$ with $\theta_*\in(0,\pi/2]$, or $\gamma_{z_0}$ intersects the line $\theta=\pi/2$ at some $s_1<0$ and then ends up converging to $(0,\theta_*)$ with $\theta_*\in(\pi/2,\pi]$.

In the first case, $\alpha_{z_0}$ is always a graph on the $x$-axis since $x'=\cos\theta$ never vanishes. In the second case, $\alpha_{z_0}$ fails to be a   graph precisely at $s=s_1$ since $x'(s_1)=0$, but it can be expressed as a vertical bi-graph whose both components are smoothly joined at $x(s_1)$. In both cases, the maximum height of $\alpha_{z_0}$ is $z_0$ at $z=0$.
\end{proof}

In this proposition, we know   that when $s$ diverges, any orbit of $\Theta$ must end up converging to some $(0,\theta_*)$. The following result restricts the possible values of $\theta_*$ and the behavior of the parameter $s$, proving that   $\alpha_{z_0}$ converges to $z=0$.

\begin{prop}\label{proporbitaborde}
Let $\gamma(s)$ be an orbit and assume that $\gamma(s)\rightarrow(0,\theta_*)$ as $s\rightarrow s_*$. Then $\theta_*=0$ and $s_*=\infty$, or $\theta_*=\pi$ and $s_*=-\infty$.
\end{prop}

\begin{proof}
Let $\gamma_{z_0}$ be the orbit passing through $(z_0,0)$ at $s=0$. Arguing by contradiction, assume that $\gamma_{z_0}$ converges to some $(0,\theta_*)$ with $\theta_*\neq0,\pi$. 

First, assume that $s$ increases from $s=0$, hence $\theta_*\in(-\pi/2,0)$. Thus $z'(s)=\sin\theta(s)\rightarrow\sin\theta_*<0$ as $s\rightarrow s_*\leq\infty$. In fact, $s_*<\infty$ since otherwise $\alpha_{z_0}$, which is arc-length parametrized, would eventually cross the line $z=0$, a contradiction. Therefore, $\alpha_{z_0}(s)\rightarrow(x_*,0)$ as $s\rightarrow s_*$, with $x_*=x(s_*)$ and $z(s_*)=0$. Fix some $\hat{x}\in(0,x_*)$ with $x_*-\hat{x}<1$ and  let $\hat{s}$ such that $\hat{x}=x(\hat{s})$. In $s\in(\hat{s},s_*)$ we write $\alpha_{z_0}(s)$    as a   graph $x\mapsto (x,0,u(x))$ with $x\in(\hat{x},x_*)$. Then, Eq. \eqref{eq1} becomes
\begin{equation}\label{eqgrafou}
\frac{u''}{1+u'^2}=-\frac{2}{u}-\frac{2u'}{u^2}.
\end{equation}
Integrating from $\hat{x}$ to $x$, we have
$$
\arctan u'(x)=-2\int_{\hat{x}}^x\frac{1}{u(t)}dt+\frac{2}{u(x)}+A(\hat{x}),\qquad A(\hat{x})=\arctan u'(\hat{x})-\frac{2}{u(\hat{x})}.
$$
By the mean value theorem, there is $c_x\in(\hat{x},x)$ such that
\begin{equation}\label{eqarc}
\arctan u'(x)=-\frac{2}{u(c_x)}(x-\hat{x})+\frac{2}{u(x)}+A(\hat{x}).
\end{equation}
 Now, let $x_n\nearrow x_*$, $x_n\in (\hat{x},x_*)$   and name $c_n=c_{x_n}$. Bearing in mind that $u(c_n)>u(x_n)$ because $u$ is strictly decreasing, we get from \eqref{eqarc}
$$
\arctan u'(x_n)=\frac{-2}{u(c_n)}(x_n-\hat{x})+\frac{2}{u(x_n)}+A(\hat{x})>\frac{-2(x_n-\hat{x})+2}{u(x_n)}+A(\hat{x}).
$$
Letting $x_n\rightarrow x_*$, we obtain
$$
\arctan u'(x_*)>\frac{-2(x_*-\hat{x})+2}{u(x_*)}+A(\hat{x}),
$$
which is a contradiction since the left-hand side is negative while the right-hand side diverges to $\infty$.   This contradiction ensures us that $\gamma_{z_0}(s)\rightarrow(0,0)$ as $s$ increases. In fact, $s\rightarrow\infty$ since otherwise $s\rightarrow s_*<\infty$ and then $\alpha_{z_0}(s)\rightarrow(x_*,0)$, arriving to the same contradiction.

To prove that $\gamma_{z_0}(s)\rightarrow(0,\pi)$ as $s\rightarrow-\infty$, we follow a similar argument. If $\theta_*\in(0,\pi/2]$ then $x'=\cos\theta\neq0$, and consequently there exists $x_*<0$ such that $\alpha_{z_0}$ is a  graph $x\mapsto(x,0,u(x))$ for $x\in(x_*,0)$, with $u''(x)<0$ and $u'(x)>0$. Note that if $\theta_*=\pi/2$, then $\alpha_{z_0}$ would intersect orthogonally $z=0$, failing to be a  graph at this intersection point. However, we just need to express $\alpha_{z_0}$ as a  graph in the open interval $x\in(x_*,0)$. An integration of \eqref{eqgrafou} from $x$ to $0$ and the mean value theorem for integrals yields
$$
\arctan u'(x)=-\frac{2x}{u(c_x)}+\frac{2}{u(x)}-\frac{2}{u(0)},\qquad c_x\in(x,0).
$$
The left-hand side is a bounded function. However,  if $x\rightarrow x_*$ the right-hand side diverges to $\infty$, a contradiction. If $\theta_*\in(\pi/2,\pi)$, now $\alpha_{z_0}$ is a  bi-graph whose lower component is again a strictly convex graph that converges to $z=0$ at a finite point. The contradiction is the same as the one exposed in the proof of $\theta_*\in(-\pi/2,0)$, concluding that $\gamma_{z_0}(s)\rightarrow(0,\pi)$ as $s\rightarrow-\infty$. This concludes the proof of Prop. \ref{proporbitbehavior}.
\end{proof}

Now we stand in position to prove the case $a\neq0$ in Th. \ref{t2}. From Props. \ref{proporbitbehavior} and \ref{proporbitaborde} we know the configuration of any orbit in the phase plane. Given $z_0>0$ there exists a unique curve $\alpha_{z_0}$  whose Euclidean height  is $z_0$ at $z=0$ and $\alpha_{z_0}$ can be expressed as a  bi-graph whose components converge to $z=0$ as $x\rightarrow\infty$. The corresponding parabolic rotational surface $\mathcal{G}(z_0)$ is a $\xi$-grim reaper fulfilling all the properties stated in Th. \ref{t2}.

\section{Non-existence of closed $\xi$-translators}\label{sec4}

In this section, we will use the properties of $\xi$-grim reapers to prove the non-existence of certain $\xi$-translators. For this,  we will employ the comparison principle of solutions of  elliptic equations for   Eq. \eqref{eq1}.   In contrast to the Euclidean space, it is not known if $\xi$-translators are minimal surfaces in a   weighted space in the sense of Ilmanen \cite{il} (a similar situation  occurs for $\chi$-translators, see   \cite{li}).   

\begin{lem}[Tangency principle]
Let $\Sigma_1$ and $\Sigma_2$ be two connected $\xi$-translators and assume that they are tangent at some $p\in \Sigma_1\cap \Sigma_2$, around. If $\Sigma_1$ lies at one side of $\Sigma_2$, then $\Sigma_1=\Sigma_2$ in the largest neighborhood of $p$ in $\Sigma_1\cap \Sigma_2$. 
\end{lem}

\begin{proof} Without loss of generality, suppose that a $\xi$-translator $\Sigma$ writes locally as $z=u(\bar{x})$, where $\bar{x}=(x,y)\in\Omega\subset\r^2$. Then Eq.  \eqref{eq1} becomes 
$$\mbox{div}\left(\frac{Du}{\sqrt{1+|Du|^2}}\right)=-\frac{1}{u^2\sqrt{1+|Du|^2}}(u+au_x+bu_y).$$
As usually, we write this equation as $\sum a_{ij}(\bar{x},u,Du)D_{ij}u+{\bf b}(\bar{x},u,Du)=0$. Here 
$${\bf b}(\bar{x},u,Du)=\frac{1+|Du|^2}{2u^2}(u+au_x+bu_y).$$
Since ${\bf b}(\bar{x},u,Du)$ is non-increasing on the variable $u$ for each fixed $(\bar{x},Du)$,  the comparison principle of quasilinear equations can be applied: see \cite[Th. 9.2]{gt}. The proof concludes by realizing that a change of the orientation keeps Eq. \eqref{eq1} invariant, hence preserves the property of being a $\xi$-translator.
\end{proof}

We apply the tangency principle to exhibit the non-existence of certain $\xi$-translators contained in Killing cylinders. Recall that  a \emph{Killing cylinder}  $\mathcal{C}_L$    around a geodesic $L$ is the set of points that lie at a fixed distance of  $L$. The convex side of $\mathcal{C}_L$ is the component of $\h^3-\mathcal{C}_L$ that contains $L$.

\begin{thm}\label{t4} If $\xi=a\partial_x+b\partial_y$, then there do not exist properly immersed $\xi$-translators in $\h^3$ without boundary contained in the convex side of a Killing cylinder $\mathcal{C}_L$. As a consequence, there are no closed (compact and without boundary) $\xi$-translators.
\end{thm}

\begin{proof}
By contradiction, suppose that $\Sigma$ is a properly immersed $\xi$-translator without boundary, contained in the convex side of a Killing cylinder $\mathcal{C}_L$. The geodesic $L$ is a vertical line or a halfcircle intersecting orthogonally the plane $z=0$. We will assume the first case, say $L$ is the $z$-axis, and a similar argument can be done if $L$ is a halfcircle.   

Consider the case $a\neq0$ in \eqref{x2}. Fix some $z_0>z_*=\inf_{p\in \Sigma}z(p)$ and let $\mathcal{G}(z_0)$ the $\xi$-grim reaper given by Th. \ref{t2} whose maximum height is $z_0$. After   parabolic translations $t\mapsto(x+t,y,z)$ with $t$ sufficiently large, the grim reaper  $\mathcal{G}(z_0)$ does not intersect the convex side of $\mathcal{C}_L$. Let us move back $\mathcal{G}(z_0)$ by  parabolic translations $t\mapsto(x+t,y,z)$ with $t\searrow-\infty$. Then there exists a first contact interior point between $\Sigma$ and $\mathcal{G}(z_0)$ because $\Sigma$ is proper and has points with height less than $z_0$.   Consequently the tangency principle  would imply that $\mathcal{G}(z_0)$ and $\Sigma$ coincide because both surfaces are complete. This is a contradiction because no $\mathcal{G}(z_0)$ is contained in the convex side of any Killing cylinder.

Suppose now  $a=0$ in \eqref{x2}. Notice that there must be points $p\in \Sigma$ with $x(p)\not=0$ since on the contrary, $\Sigma$ would be with the vertical plane $x=0$, which it is not  contained in the convex side of $\mathcal{C}_L$. Let $p^*\in\Sigma$ be a point with $x(p^*)\neq0$, say $x(p^*)>0$, and let $x_1= x(p^*)/2$. Consider the  solutions $z_\lambda=z_\lambda(x)$ of \eqref{eq-p1} where $z(x_1)=z(x_1+\lambda)=0$, $\lambda>0$, in particular,   $z'(\frac{x_1+\lambda}{2})=0$. If $\lambda$ is sufficiently small, then the graphic of $z_\lambda$ is outside $\mathcal{C}_L$, or equivalently, $\mathcal{G}^0(z(x_\lambda))\cap \mathcal{C}_L=\emptyset$. Letting  $\lambda\nearrow\infty$  the graph of $z_\lambda$ is asymptotic to the vertical plane of equation $x=x_1$ (see Rem. \ref{remark2};  also   \cite{dl}). Since $x(p^*)>x_1$, there  there is a first time $\lambda_0$ such that  $\mathcal{G}^0(z(x_{\lambda_0}))$ touches $\Sigma$.    The tangency principle implies that both surfaces coincide,  a contradiction.

\end{proof}
 
\section{Appendix: grim reapers in $\h^3$}\label{sec5}

In hyperbolic space there are different possible vector fields $X$ to consider in   \eqref{eq1} as well as notions of grim reapers. Due to this variety of situations, in this section we summarize the recent developments achieved in the literature about grim reapers of $\h^3$. If $X$ is a Killing vector field, there are the choices $\xi$ and $\chi$ studied in this paper, whose flows of isometries are, respectively, the parabolic and hyperbolic translations of $\h^3$. 
Another vector fields of interest  are the conformal Killing vector fields, as for example the vector fields $\partial_z$, \cite{bl}, and $-\partial_z$, \cite{mr}. These vector fields are of special interest because  $\partial_z$-translators and $-\partial_z$-translators are the analogues of the Euclidean self-shrinkers and self-expanders, respectively. In contrast to the vector fields $\xi$ and $\chi$,  $\partial_z$-translators and $-\partial_z$-translators are minimal surfaces in a density space in the sense of Ilmanen \cite{il}. Indeed, if   $\h^3$ is endowed with the conformal metric $e^{- 2/z}\langle,\rangle$ (resp. $e^{ 2/z}\langle,\rangle$), then minimal surfaces in this conformal space are just $\partial_z$-translators (resp. $-\partial_z$-translators) in $\h^3$. 

To precise the notion of grim reaper, and following with the Euclidean ambient space, an $X$-grim reaper is an $X$-translator which is invariant by parabolic or hyperbolic translations. To simplify the notation, we write $(P,X)$-grim reaper or $(H,X)$-grim reaper, respectively.  We summarize all types of grim reapers in $\h^3$.

\begin{thm} The $X$-grim reapers in $\h^3$ are the following.
\begin{enumerate} 
\item $(H,\xi)$-grim reapers and $(P,\xi)$-grim reapers have been classified in this paper: see Thms. \ref{t1} and \ref{t2} respectively. 
\item $(H,\chi)$-grim reapers and $(P,\chi)$-grim reapers are classified in \cite{li}.  
\item $(H,\partial_z)$-grim reapers and $(P,\partial_z)$-grim reapers are classified in  \cite{bl}.  
\item   $(P,-\partial_z)$  and $(H,-\partial_z)$-grim reapers are classified in  \cite{bl} and   \cite{mr}, respectively.
\end{enumerate}
\end{thm}


\section*{Acknowledgment}

Antonio Bueno has been partially supported by CARM, Programa Regional de Fomento de la Investigaci\'on, Fundaci\'on S\'eneca-Agencia de Ciencia y Tecnolog\'{\i}a Regi\'on de Murcia, reference 21937/PI/22. 

Rafael L\'opez  is a member of the IMAG and of the Research Group ``Problemas variacionales en geometr\'{\i}a'',  Junta de Andaluc\'{\i}a (FQM 325). This research has been partially supported by PID2020-117868GB-I00,  and by the ``Mar\'{\i}a de Maeztu'' Excellence Unit IMAG, reference CEX2020-001105- M, funded by MCINN/AEI/10.13039/501100011033/CEX2020-001105-M.

\end{document}